\documentclass[11pt]{amsart}
\usepackage[margin=1.05in]{geometry}
\usepackage[T1]{fontenc}
\usepackage{lmodern}
\usepackage{amsmath,amssymb,amsthm,mathtools,array,booktabs,microtype}
\usepackage[hidelinks]{hyperref}
\hypersetup{pdftitle={Diagonal parity and loop toggling for symmetric matrices over F2},pdfauthor={Mohsen Aliabadi}}
\newtheorem{theorem}{Theorem}[section]
\newtheorem{proposition}[theorem]{Proposition}
\newtheorem{lemma}[theorem]{Lemma}
\newtheorem{corollary}[theorem]{Corollary}
\theoremstyle{definition}
\newtheorem{definition}[theorem]{Definition}
\newtheorem{example}[theorem]{Example}
\theoremstyle{remark}
\newtheorem{remark}[theorem]{Remark}
\DeclareMathOperator{\diag}{diag}
\DeclareMathOperator{\rank}{rank}
\DeclareMathOperator{\nul}{null}
\DeclareMathOperator{\Img}{Im}
\DeclareMathOperator{\supp}{supp}
\newcommand{\F}{\mathbb F_2}
\newcommand{\one}{\mathbf 1}
\newcommand{\eps}{\varepsilon}
\numberwithin{equation}{section}
\title[Diagonal parity and loop toggling]{Diagonal parity and loop toggling for symmetric matrices over $\mathbb F_2$}
\author{Mohsen Aliabadi}
\address{Department of Mathematics, Clayton State University, Morrow, GA 30260, USA}

\email{mohsenmath88@gmail.com}
\date{}
\begin{document}
\begin{abstract}
For a symmetric matrix $M$ over $\mathbb F_2$, it is known that its diagonal vector belongs to its column space. We prove the parity refinement that every solution of $Mx=\diag(M)$ satisfies $\diag(M)^Tx\equiv\rank(M)\pmod 2$. This extends the parity theorem for odd dominating sets to graph matrices with arbitrary binary diagonal labels. A self-contained rank-one update formula gives the changes caused by a single loop toggle and yields
\[
\nul A(\overline G)=\nul(A(G)+I)+[\rank(A(G)+I)\text{ is odd}].
\]
For rooted trees, we augment Filmus's three classes of realizable boundary data with an explicit multiplicity recursion. The resulting formula computes the nullity of $A(T)+D_\eps$ in a single traversal and gives a nonsingularity criterion. For complete rooted trees with eventually periodic labels indexed by height above the leaves, we obtain eventual periodicity of the boundary states and nullity formulas affine in powers of the branching degree on each eventual residue class. Constant labels give explicit closed forms.
\end{abstract}
\maketitle

\section{Introduction}\label{sec:intro}

Throughout, all matrices and vector spaces are over $\F$. For a square matrix $M$, let $\diag(M)$ denote the column vector of its diagonal entries, let $\Img(M)$ be its column space, and write $\rank(M)=\dim\Img(M)$ and $\nul(M)=\dim\ker M$. The identity matrix of order $m$ is $I_m$, or $I$ when its order is clear. For a finite set $V$, vectors in $\F^V$ have coordinates indexed by $V$; $\one$ is the all-ones vector and $e_v$ is the coordinate vector at $v$. For $u\in\F^V$, let $D_u$ be the diagonal matrix with diagonal $u$. For a statement $Q$, the integer $[Q]$ is $1$ if $Q$ holds and $0$ otherwise. We identify the elements of $\F$ with $0$ and $1$ when writing congruences or integer counts.

Let $G$ be a finite simple graph. Its adjacency matrix $A(G)$ has zero diagonal, and $w\sim v$ means that $w$ and $v$ are adjacent. The open and closed neighborhoods of $v$ are
\[
N_G(v)=\{w:w\sim v\},\qquad N_G[v]=N_G(v)\cup\{v\}.
\]
A set $S\subseteq V(G)$ is \emph{odd dominating} if $|S\cap N_G[v]|$ is odd for every vertex $v$. Equivalently, its characteristic vector $x$ satisfies
\[
(A(G)+I)x=\one.
\]
Sutner proved that this system is always solvable \cite{Sutner1989}. Its solution space and the nullity of the closed-neighborhood matrix are central to parity domination; the latter is called the \emph{parity dimension} of $G$ \cite{AminClarkSlater1998,AminSlaterZhang2002}. Related matrix methods occur in switch-setting problems and Lights Out \cite{Goldwasser1997,Minevich2012}. Batal \cite[Theorem~1]{Batal2022} proved that every odd dominating set has cardinality congruent to $\rank(A(G)+I)$ modulo $2$.

The solvability property
\[
\diag(M)\in\Img(M)
\]
is already known for every symmetric binary matrix, with no assumption that its diagonal is all ones; see Filmus \cite{Filmus2010} and Minevich \cite{Minevich2012}. We recall it in Proposition~\ref{prop:diag-range}, using Alon's argument recorded by Filmus. The parity refinement proved here is Theorem~\ref{thm:parity}: every solution of $Mx=\diag(M)$ satisfies
\[
\diag(M)^Tx\equiv\rank(M)\pmod 2.
\]
The known solvability property and the parity refinement are stated separately. The proof of the latter uses induction and two explicit congruence transformations. Lemma~\ref{lem:congruence} records the behavior of the diagonal under congruence, and Corollary~\ref{cor:inverse} gives the inverse-diagonal identity in the nonsingular case.

For graph matrices, an arbitrary binary diagonal is a choice of loops. If $\eps\in\F^{V(G)}$, put
\[
M(G,\eps)=A(G)+D_\eps.
\]
The equation $M(G,\eps)x=\eps$ imposes odd parity on the closed neighborhood at a vertex labeled $1$, and even parity on the open neighborhood at a vertex labeled $0$. Corollary~\ref{cor:graph} gives the corresponding solvability and parity statements.

We then consider changes in the matrix and the number of solutions. Rank changes under rank-one modifications have a standard linear-algebraic description; see, for example, Meyer \cite{Meyer2018} for real and complex matrices. Theorem~\ref{thm:update} gives a direct proof in the symmetric binary setting. The update $M+uu^T$ toggles both loops and edges within the support of $u$, whereas toggling only those loops is the diagonal update $M+D_u$. Corollary~\ref{cor:loop} treats a single loop. Combining the rank-one formula with diagonal parity gives the complementation identity in Corollary~\ref{cor:complement}.

Recursive methods for tree parity dimension were developed by Wagner \cite{Wagner2010}, and Hatzl and Wagner \cite{HatzlWagner2008} studied the more general problem with independently prescribed open or closed neighborhoods and even or odd requirements. For the right-hand side tied to the diagonal, Filmus \cite[Theorem~3.1]{Filmus2010} classified the possible boundary pairs into three affine lines. Section~\ref{sec:trees} recalls this classification and augments it with an explicit formula for the number of solutions with each realizable pair. Theorem~\ref{thm:tree} leads to a sum formula for nullity, a nonsingularity criterion, and a worked example (Corollaries~\ref{cor:nullity} and~\ref{cor:nonsingular}, Example~\ref{ex:tree}). In Section~\ref{sec:complete}, the same recursion becomes a three-state system for complete rooted trees. It gives eventual nullity formulas for periodic labels and closed forms for constant labels (Theorem~\ref{thm:periodic} and Corollary~\ref{cor:constant}).

\section{The diagonal vector and the parity theorem}\label{sec:parity}

For a subspace $W\subseteq\F^n$, write
\[
W^\perp=\{z\in\F^n:z^Tw=0\text{ for every }w\in W\}.
\]
For a symmetric matrix $M\in\F^{n\times n}$, set
\[
\mathcal S(M)=\{x\in\F^n:Mx=\diag(M)\}.
\]

\begin{proposition}[Diagonal-range property {\cite[Section~4]{Filmus2010}, \cite{Minevich2012}}]\label{prop:diag-range}
If $M\in\F^{n\times n}$ is symmetric, then $\diag(M)\in\Img(M)$. Consequently, $\mathcal S(M)$ is a nonempty affine subspace of dimension $n-\rank(M)$.
\end{proposition}
\begin{proof}
We use Alon's proof as recorded by Filmus. For every $z\in\F^n$, symmetry and characteristic two give
\begin{equation}\label{eq:quadratic}
z^TMz=\sum_i M_{ii}z_i^2+\sum_{i<j}(M_{ij}+M_{ji})z_iz_j
=\diag(M)^Tz.
\end{equation}
Thus $\diag(M)^Tz=0$ for every $z\in\ker M$. For any matrix $M$, orthogonality to all columns is equivalent to $M^Tz=0$, so $(\Img(M))^\perp=\ker M^T$. Taking orthogonal complements and using symmetry yields
\begin{equation}\label{eq:orthogonal}
\Img(M)=(\ker M)^\perp.
\end{equation}
Hence $\diag(M)\in\Img(M)$. If $Mx_0=\diag(M)$, then $\mathcal S(M)=x_0+\ker M$, and rank-nullity gives its dimension.
\end{proof}

\begin{theorem}[Diagonal parity]\label{thm:parity}
Let $M\in\F^{n\times n}$ be symmetric. Every $x\in\mathcal S(M)$ satisfies
\begin{equation}\label{eq:parity}
\diag(M)^Tx\equiv\rank(M)\pmod 2.
\end{equation}
\end{theorem}
\begin{proof}
Put $d=\diag(M)$ and fix a solution $x$ of $Mx=d$. We induct on $n$, allowing the empty matrix. The assertion is immediate when $M=0$. Suppose that $M\ne0$ and that the assertion holds for all symmetric matrices of smaller order.

\emph{Case 1: $M$ has a nonzero diagonal entry.}
After simultaneously permuting the rows, columns, and coordinates of $x$, we may assume that $M_{11}=1$. This preserves both $d^Tx$ and rank. Write
\[
M=\begin{pmatrix}1&b^T\\b&C\end{pmatrix},\qquad
x=\begin{pmatrix}\alpha\\y\end{pmatrix}.
\]
The equation $Mx=d$ says
\[
\alpha+b^Ty=1,\qquad b\alpha+Cy=\diag(C).
\]
Eliminating $\alpha$ gives $(C+bb^T)y=\diag(C)+b$. Set $R=C+bb^T$. Since $b_i^2=b_i$, we have $\diag(R)=\diag(C)+b$, and hence $Ry=\diag(R)$. By induction,
\[
d^Tx=\alpha+\diag(C)^Ty
=1+\diag(R)^Ty\equiv1+\rank(R)\pmod 2.
\]
The invertible matrix
\[
P=\begin{pmatrix}1&b^T\\0&I_{n-1}\end{pmatrix}
\quad\text{satisfies}\quad
P^TMP=\begin{pmatrix}1&0\\0&R\end{pmatrix}.
\]
Therefore $\rank(M)=1+\rank(R)$, as required.

\emph{Case 2: $M$ has zero diagonal.}
Here $d=0$, so we must show that $\rank(M)$ is even. Since $M\ne0$, an off-diagonal entry is nonzero. A simultaneous row and column permutation gives
\[
M=\begin{pmatrix}H&E\\E^T&C\end{pmatrix},\qquad
H=\begin{pmatrix}0&1\\1&0\end{pmatrix},
\]
where $C$ is symmetric with zero diagonal. Since $H^2=I_2$, the invertible matrix
\[
Q=\begin{pmatrix}I_2&HE\\0&I_{n-2}\end{pmatrix}
\quad\text{satisfies}\quad
Q^TMQ=\begin{pmatrix}H&0\\0&R\end{pmatrix},
\qquad R=C+E^THE.
\]
The matrix $R$ is symmetric with zero diagonal: if $f_j$ is a column of $E$, then $(E^THE)_{jj}=f_j^THf_j=0$. Apply the induction hypothesis to $R$ and the solution $y=0$ of $Ry=\diag(R)=0$. It follows that $\rank(R)$ is even. Thus $\rank(M)=2+\rank(R)$ is even, completing the proof.
\end{proof}

\begin{lemma}[Diagonal under congruence]\label{lem:congruence}
If $M\in\F^{n\times n}$ is symmetric and $P\in\F^{n\times m}$, then
\[
\diag(P^TMP)=P^T\diag(M).
\]
\end{lemma}
\begin{proof}
If $p_j$ is column $j$ of $P$, the $j$th diagonal entry of $P^TMP$ is $p_j^TMp_j=\diag(M)^Tp_j$ by \eqref{eq:quadratic}.
\end{proof}

\begin{corollary}[Inverse diagonal]\label{cor:inverse}
If $M\in\F^{n\times n}$ is symmetric and invertible, then
\[
M^{-1}\diag(M)=\diag(M^{-1}),\qquad
\diag(M)^T\diag(M^{-1})\equiv n\pmod 2.
\]
\end{corollary}
\begin{proof}
Apply Lemma~\ref{lem:congruence} with $P=M^{-1}$. Since $M^{-1}$ is symmetric and $P^TMP=M^{-1}$, the first identity follows. Thus $x=\diag(M^{-1})$ is the unique solution of $Mx=\diag(M)$, and Theorem~\ref{thm:parity} gives the second identity.
\end{proof}

When $\diag(M)=\one$, the second identity recovers the parity of the sum of the entries of $M^{-1}$ in Batal's lemma \cite[Lemma~1]{Batal2022}: off-diagonal entries cancel in pairs. The scalar identity also follows directly by expanding $\operatorname{tr}(MM^{-1})$ and pairing off-diagonal terms.

\begin{remark}[Characteristic vectors]\label{rem:characteristic}
For the symmetric bilinear form $B(v,w)=v^TMw$, a vector $x$ is called characteristic if $B(v,x)=B(v,v)$ for every $v$. Equation~\eqref{eq:quadratic} shows that these are precisely the vectors in $\mathcal S(M)$. Theorem~\ref{thm:parity} states that every such vector satisfies $B(x,x)\equiv\rank(M)\pmod2$. This interpretation includes degenerate forms and does not require a choice of normal form.
\end{remark}

\section{Graph parity and symmetric rank-one updates}\label{sec:updates}

Let $G$ be a finite simple graph and $\eps\in\F^{V(G)}$. Recall that $M(G,\eps)=A(G)+D_\eps$ records a loop at exactly those vertices with label $1$.

\begin{definition}\label{def:pattern}
A vector $x\in\F^{V(G)}$ is an \emph{$\eps$-odd dominating pattern} if $M(G,\eps)x=\eps$.
\end{definition}

In coordinates, the defining equation is
\[
\sum_{w\sim v}x_w+\eps_vx_v=\eps_v\qquad(v\in V(G)).
\]
Thus the selected set has odd intersection with $N_G[v]$ when $\eps_v=1$, and even intersection with $N_G(v)$ when $\eps_v=0$.

\begin{corollary}\label{cor:graph}
For every finite simple graph $G$ and every binary label $\eps$, the set of $\eps$-odd dominating patterns is a nonempty affine subspace of dimension $\nul M(G,\eps)$. Every such pattern satisfies
\[
\eps^Tx\equiv\rank M(G,\eps)\pmod2.
\]
\end{corollary}
\begin{proof}
Apply Proposition~\ref{prop:diag-range} and Theorem~\ref{thm:parity} to the symmetric matrix $M(G,\eps)$, whose diagonal vector is $\eps$.
\end{proof}

The number of patterns is $2^{\nul M(G,\eps)}$, and $\eps^Tx$ is the parity of the number of selected looped vertices. For $\eps=\one$, this recovers ordinary odd domination and its parity theorem.

\emph{Toggling} an edge or loop means deleting it if present and adding it if absent. For $u\in\F^{V(G)}$, let $\supp(u)=\{v:u_v=1\}$. Since $(uu^T)_{vw}=u_vu_w$, the update $M+uu^T$ toggles every loop in $\supp(u)$ and every edge between distinct vertices of this support. Toggling only the loops is instead $M+D_u$. For a single vertex $v$, both updates are $M+e_ve_v^T$.

The next result is a self-contained binary rank-one update formula. Its role here is to connect the parity theorem with the two graph operations in Corollaries~\ref{cor:loop} and~\ref{cor:complement}; no novelty is claimed for the general principle governing rank-one modifications.

\begin{theorem}[Symmetric rank-one update]\label{thm:update}
Let $M\in\F^{n\times n}$ be symmetric and $u\in\F^n$. If $u=0$, the matrix is unchanged. If $u\ne0$, put
\[
\Delta_r=\rank(M+uu^T)-\rank(M),\qquad
\Delta_\nu=\nul(M+uu^T)-\nul(M).
\]
When $u\in\Img(M)$, the scalar $u^Ty$ is independent of the solution $y$ of $My=u$. The changes are
\[
\begin{array}{c|rr}
\text{condition}&\Delta_r&\Delta_\nu\\\hline
u\notin\Img(M)&1&-1\\
u\in\Img(M),\ u^Ty=0&0&0\\
u\in\Img(M),\ u^Ty=1&-1&1
\end{array}
\]
\end{theorem}
\begin{proof}
Suppose first that $u\notin\Img(M)$. By \eqref{eq:orthogonal}, some $z\in\ker M$ satisfies $u^Tz=1$. If $(M+uu^T)x=0$, multiplying by $z^T$ gives $u^Tx=0$, and then $Mx=0$. Conversely, these two equations imply $(M+uu^T)x=0$. Hence
\[
\ker(M+uu^T)=\{x\in\ker M:u^Tx=0\}.
\]
The functional $x\mapsto u^Tx$ is nonzero on $\ker M$, so this subspace has codimension one. This proves the first row of the table.

Now suppose that $u\in\Img(M)$ and choose $y$ with $My=u$. The difference of any two choices belongs to $\ker M$ and is orthogonal to $u$, proving that $u^Ty$ is well defined. Writing $t=u^Tx$, the equation $(M+uu^T)x=0$ gives
\[
M(x+ty)=0,\qquad x=z+ty\quad(z\in\ker M).
\]
Conversely, for $z\in\ker M$ and $t\in\F$, the relation $t=u^T(z+ty)$ is exactly $t=t\,u^Ty$, since $u^Tz=0$. If $u^Ty=0$, then $t=0$ and the kernel is unchanged. If $u^Ty=1$, then $t$ is free; moreover, $y\notin\ker M$ because $My=u\ne0$. Thus the kernel gains one dimension. Rank-nullity gives the rank changes.
\end{proof}

\begin{corollary}[A single loop]\label{cor:loop}
Let $M=M(G,\eps)$ and $M'=M(G,\eps+e_v)$. Then
\[
\nul(M')-\nul(M)=
\begin{cases}
-1,&e_v\notin\Img(M),\\
0,&e_v\in\Img(M)\text{ and }y_v=0,\\
1,&e_v\in\Img(M)\text{ and }y_v=1,
\end{cases}
\]
where $My=e_v$ in the last two cases. The rank change is the negative of the nullity change.
\end{corollary}
\begin{proof}
Use $M'=M+e_ve_v^T$ and $e_v^Ty=y_v$ in Theorem~\ref{thm:update}.
\end{proof}

Toggling several loops can be treated by applying this corollary successively to the current matrix at each step. It is not represented by a single update $uu^T$ unless $|\supp(u)|\le1$.

\begin{corollary}[Complementation]\label{cor:complement}
For every finite simple graph $G$,
\begin{equation}\label{eq:complement}
\nul A(\overline G)=\nul(A(G)+I)+[\rank(A(G)+I)\text{ is odd}].
\end{equation}
Consequently,
\[
|\ker A(\overline G)|=
2^{[\rank(A(G)+I)\text{ is odd}]}
\bigl|\{x:(A(G)+I)x=\one\}\bigr|.
\]
\end{corollary}
\begin{proof}
The empty graph on no vertices is immediate. Otherwise put $N=A(G)+I$. Over $\F$,
\[
A(\overline G)=N+\one\one^T.
\]
Proposition~\ref{prop:diag-range} gives $\one\in\Img(N)$, and Theorem~\ref{thm:parity} gives $\one^Ty\equiv\rank(N)\pmod2$ for every solution of $Ny=\one$. The last two rows of Theorem~\ref{thm:update} prove \eqref{eq:complement}. Taking powers of $2$ gives the counting identity.
\end{proof}

More generally, $M(G,\eps)+\one\one^T=M(\overline G,\eps+\one)$: complementation toggles all edges and all loops. In \eqref{eq:complement}, the initial diagonal is all ones, so the final graph has no loops.

\section{Rooted trees with arbitrary diagonal labels}\label{sec:trees}

A rooted tree is a finite tree with a distinguished vertex $r$, its root. The depth of a vertex is its distance from $r$, and the height of the tree is the largest depth. A child is an adjacent vertex of greater depth; a leaf has no children. For a vertex $v$, let $T_v$ be the subtree induced by $v$ and its descendants, rooted at $v$.

Let $T$ have binary labeling $\eps:V(T)\to\F$, also regarded as a vector, and put $M(T,\eps)=A(T)+D_\eps$. Write $T_1,\ldots,T_m$ for the child subtrees of $r$, with roots $r_1,\ldots,r_m$ and restricted labels $\eps_i=\eps|_{V(T_i)}$. For $\alpha,\beta\in\F$, define
\begin{equation}\label{eq:boundary-count}
N_T(\alpha,\beta)=\#\{x\in\F^{V(T)}:M(T,\eps)x=\eps+\alpha e_r,\ x_r=\beta\}.
\end{equation}
Here $\alpha$ allows a defect in the root equation and $\beta$ specifies the root value. Let
\[
L_T=\{(\alpha,\beta):N_T(\alpha,\beta)>0\}.
\]
For $L\subseteq\F^2$, let $\chi_L$ be its integer-valued indicator. The three lines used below are
\begin{equation}\label{eq:lines}
\ell_0=\{(\alpha,\beta):\alpha+\beta=1\},\quad
\ell_1=\{(\alpha,\beta):\alpha=0\},\quad
\ell_2=\{(\alpha,\beta):\beta=0\}.
\end{equation}

\begin{theorem}[Three boundary classes and their multiplicities]\label{thm:tree}
For every rooted binary-labeled tree $T$, the set $L_T$ is one of the three lines in \eqref{eq:lines}, and there is an integer $k_T\ge0$ such that
\begin{equation}\label{eq:uniform-count}
N_T(\alpha,\beta)=2^{k_T}\chi_{L_T}(\alpha,\beta).
\end{equation}
The classification of $L_T$ is Filmus's \cite[Theorem~3.1]{Filmus2010}. The following recursion also gives the multiplicities. Put
\[
a=\#\{i:L_{T_i}=\ell_1\},\qquad b=\#\{i:L_{T_i}=\ell_0\}.
\]
Then
\begin{equation}\label{eq:tree-transition}
L_T=
\begin{cases}
\ell_2,&a>0,\\
\ell_1,&a=0\text{ and }\eps(r)+b\equiv0\pmod2,\\
\ell_0,&a=0\text{ and }\eps(r)+b\equiv1\pmod2,
\end{cases}
\end{equation}
and
\begin{equation}\label{eq:k-recursion}
k_T=\sum_{i=1}^m k_{T_i}+\max(a-1,0).
\end{equation}
For a leaf, the sum is empty, $k_T=0$, and $L_T=\ell_1$ for label $0$ and $L_T=\ell_0$ for label $1$.
\end{theorem}
\begin{proof}
We induct on the height. For a single vertex, the equation is $\eps(r)\beta=\eps(r)+\alpha$. Each admissible pair determines one vector. This gives the stated initial conditions.

Suppose now that the conclusions hold for the child subtrees. Fix $x_r=\beta$ and put $\gamma_i=x_{r_i}$. On $T_i$, the parent contributes only at $r_i$, so
\[
M(T_i,\eps_i)x|_{V(T_i)}=\eps_i+\beta e_{r_i}.
\]
The coordinate vector in this equation is indexed by $V(T_i)$. The root equation is
\begin{equation}\label{eq:root-eqn}
\sum_{i=1}^m\gamma_i+\eps(r)\beta=\eps(r)+\alpha.
\end{equation}
For each admissible choice of $\gamma_1,\ldots,\gamma_m$, the number of extensions to all child subtrees is $2^{\sum_i k_{T_i}}$. It remains to count such choices.

If $a>0$, a child of class $\ell_1$ requires $\beta=0$. With this value fixed, each of the $a$ children of class $\ell_1$ has a free root value $\gamma_i$, each child of class $\ell_0$ has $\gamma_i=1$, and each child of class $\ell_2$ has $\gamma_i=0$. Equation~\eqref{eq:root-eqn} is one equation in the $a$ free variables, with every coefficient equal to $1$. For either value of $\alpha$, it has exactly $2^{a-1}$ solutions. Thus $L_T=\ell_2$ and $k_T=\sum_i k_{T_i}+a-1$.

If $a=0$, every child root value is determined by $\beta$: it is $1+\beta$ for a child of class $\ell_0$ and $0$ for one of class $\ell_2$. Equation~\eqref{eq:root-eqn} therefore becomes
\[
\alpha=(\eps(r)+b)(1+\beta)
\]
over $\F$. If $\eps(r)+b$ is even, this is $\alpha=0$; if it is odd, this is $\alpha+\beta=1$. Each admissible pair has exactly one choice of the child root values. Hence $k_T=\sum_i k_{T_i}$, proving all assertions.
\end{proof}

\begin{corollary}[A sum formula for nullity]\label{cor:nullity}
For each vertex $v$, let $a_v$ be the number of children $w$ of $v$ for which $L_{T_w}=\ell_1$. Then
\begin{equation}\label{eq:nullity-sum}
\nul M(T,\eps)=\sum_{v\in V(T)}\max(a_v-1,0)+[L_T=\ell_1].
\end{equation}
The states and the nullity can be computed in $O(|V(T)|)$ arithmetic operations on integers of $O(\log(|V(T)|+1))$ bits.
\end{corollary}
\begin{proof}
The solutions of $M(T,\eps)x=\eps$ are counted by $N_T(0,0)+N_T(0,1)$. The line $\ell_1$ contains both pairs with $\alpha=0$, whereas each of $\ell_0$ and $\ell_2$ contains exactly one. Proposition~\ref{prop:diag-range} and \eqref{eq:uniform-count} consequently give
\begin{equation}\label{eq:nullity-state}
\nul M(T,\eps)=k_T+[L_T=\ell_1].
\end{equation}
Iterating \eqref{eq:k-recursion} gives the sum in \eqref{eq:nullity-sum}. A traversal from leaves to root computes each state from its children's states. Each edge is processed once, and all counts and exponents are bounded by $|V(T)|$.
\end{proof}

\begin{corollary}[Nonsingularity]\label{cor:nonsingular}
The matrix $M(T,\eps)$ is nonsingular if and only if $a_v\le1$ for every vertex $v$ and $L_T\ne\ell_1$. These conditions are also equivalent to the existence of exactly one $\eps$-odd dominating pattern.
\end{corollary}
\begin{proof}
Every term on the right of \eqref{eq:nullity-sum} is nonnegative, so the sum is zero exactly under the stated conditions. The solution set is nonempty and has $2^{\nul M(T,\eps)}$ elements.
\end{proof}

\begin{example}[A mixed labeling]\label{ex:tree}
Let $T$ have vertices $r,u,v,p,q$, root $r$, and edges $ru,rv,up,uq$. Give $r,u,v$ label $1$ and $p,q$ label $0$. The recursion gives
\[
\begin{array}{c|c|c|c|c}
\text{vertex }w&\eps(w)&a_w&L_{T_w}&k_{T_w}\\\hline
p,q&0&0&\ell_1&0\\
v&1&0&\ell_0&0\\
u&1&2&\ell_2&1\\
r&1&0&\ell_1&1
\end{array}
\]
Thus $\nul M(T,\eps)=1+1=2$ and $\rank M(T,\eps)=3$. In the vertex order $(r,u,v,p,q)$, the solutions are exactly
\[
x=(t,0,1+t,s,1+t+s)^T,\qquad s,t\in\F.
\]
Indeed, the equations at $p,q$ give $x_u=0$, the equation at $v$ gives $x_v=1+x_r$, and that at $u$ gives $x_q=1+x_r+x_p$; the equation at $r$ is then automatic. There are four solutions, and every one satisfies $\eps^Tx=t+(1+t)=1$, in agreement with Theorem~\ref{thm:parity}.
\end{example}

\begin{remark}[Relation to earlier tree results]\label{rem:filmus}
The three lines and their transition rule are already present in Filmus's proof. The counting refinement in \eqref{eq:k-recursion} records the $a-1$ free choices when $a>0$, and turns the realizability recursion into the nullity sum \eqref{eq:nullity-sum}. For $\eps=\one$, this nullity is the parity dimension studied in \cite{AminClarkSlater1998,AminSlaterZhang2002}, for which recursive computation on trees is also established \cite{Wagner2010}. The formulation here keeps the boundary multiplicities for an arbitrary fixed binary labeling. It differs from counting solvable instances over all choices of neighborhood and parity constraints, as in \cite{HatzlWagner2008}. In particular, three realizability classes suffice even with arbitrary labels. For example, the path on two vertices with both labels $1$, rooted at an endpoint, has count vector $(N_T(0,0),N_T(0,1),N_T(1,0),N_T(1,1))=(1,1,0,0)$ and belongs to class $\ell_1$.
\end{remark}

\section{Complete rooted trees and periodic labels}\label{sec:complete}

For integers $d\ge1$ and $h\ge0$, let $T_h^{(d)}$ be the complete rooted $d$-ary tree of height $h$: every vertex of depth less than $h$ has exactly $d$ children, and the vertices at depth $h$ are leaves.

\begin{definition}\label{def:height-label}
For a binary sequence $a=(a_0,a_1,\ldots)$, its associated labeling of $T_h^{(d)}$ is
\[
\eps_h(v)=a_{h-\operatorname{depth}(v)}.
\]
\end{definition}

The index is height above the leaves: leaves have label $a_0$, their parents label $a_1$, and the root label $a_h$. Each child subtree of $T_{h+1}^{(d)}$ therefore has precisely the labeling $\eps_h$. Write $L_h,k_h$ for its boundary state and exponent, and put $\nu_h=\nul M(T_h^{(d)},\eps_h)$. Let $\bar d\in\F$ be the parity of $d$, and set $\Lambda(0)=\ell_1$, $\Lambda(1)=\ell_0$.

\begin{proposition}[Explicit finite-state recursion]\label{prop:complete}
The initial conditions are $L_0=\Lambda(a_0)$ and $k_0=0$. For $h\ge0$, the transition is
\begin{equation}\label{eq:complete-table}
\begin{array}{c|c|c}
L_h&L_{h+1}&k_{h+1}\\\hline
\ell_0&\Lambda(a_{h+1}+\bar d)&d k_h\\
\ell_1&\ell_2&d k_h+d-1\\
\ell_2&\Lambda(a_{h+1})&d k_h
\end{array}
\end{equation}
where the addition inside $\Lambda$ is in $\F$. In particular,
\begin{equation}\label{eq:complete-k}
k_{h+1}=d k_h+(d-1)[L_h=\ell_1],\qquad
\nu_h=k_h+[L_h=\ell_1].
\end{equation}
\end{proposition}
\begin{proof}
All $d$ child subtrees have the same state. In Theorem~\ref{thm:tree}, their counts $(a,b)$ are respectively $(0,d)$, $(d,0)$, and $(0,0)$ for the three rows. Equations~\eqref{eq:tree-transition}, \eqref{eq:k-recursion}, and~\eqref{eq:nullity-state} give the formulas.
\end{proof}

A sequence $(a_j)$ is eventually periodic if there exist $j_0\ge0$ and $q\ge1$ such that $a_{j+q}=a_j$ for every $j\ge j_0$.

\begin{theorem}[Eventually periodic labels]\label{thm:periodic}
Let $d\ge2$ and let $a$ be eventually periodic. Then $(L_h)$ is eventually periodic. There exist integers $h_0\ge0$, $p\ge1$, and rational numbers $c_r,b_r$ for $0\le r<p$ such that
\begin{equation}\label{eq:periodic-nullity}
\nu_h=c_r d^h+b_r\qquad(h\ge h_0,\ h\equiv r\pmod p).
\end{equation}
If $q$ is an eventual period of the labels, $p$ may be chosen to be a multiple of $q$ with $p\le3q$.
\end{theorem}
\begin{proof}
After the labels become periodic, the pair $(L_h,h\bmod q)$ follows a deterministic map on a set of $3q$ elements, by \eqref{eq:complete-table}. Every orbit eventually enters a cycle. Let $p$ be its cycle length. Then $p\le3q$, and $q$ divides $p$ because the second coordinate advances by one at each step. Thus the states and the labels have common eventual period $p$.

Set $s_h=(d-1)[L_h=\ell_1]$. Iterating \eqref{eq:complete-k} for $p$ steps gives
\[
k_{h+p}=d^p k_h+C_h,\qquad
C_h=\sum_{j=0}^{p-1}d^{p-1-j}s_{h+j}.
\]
For sufficiently large $h$, $C_h$ is constant on each residue class modulo $p$. Fix a class $r$ and a sufficiently large $h_r\equiv r\pmod p$, and write $C_r=C_{h_r}$ and $\delta_r=[L_{h_r}=\ell_1]$. Then, for every $t\ge0$,
\[
k_{h_r+tp}=d^{tp}\left(k_{h_r}+\frac{C_r}{d^p-1}\right)-\frac{C_r}{d^p-1}.
\]
Since $d\ge2$, the denominator is nonzero. Adding $\delta_r$ proves \eqref{eq:periodic-nullity}, with
\[
c_r=d^{-h_r}\left(k_{h_r}+\frac{C_r}{d^p-1}\right),\qquad
b_r=\delta_r-\frac{C_r}{d^p-1}.
\]
Taking $h_0$ large enough for all residue classes completes the proof.
\end{proof}

\begin{corollary}[Constant labels]\label{cor:constant}
Let $d\ge1$, $h\ge0$, and let $\nu_h^{(e)}=\nul(A(T_h^{(d)})+eI)$ for $e\in\{0,1\}$. Then
\begin{equation}\label{eq:zero-label}
\nu_h^{(0)}=\frac{d^{h+1}+(-1)^h}{d+1}.
\end{equation}
For labels all equal to $1$, if $d$ is even, then $\nu_h^{(1)}=0$ for every $h$. If $d$ is odd and $r\in\{0,1,2\}$ is the residue of $h$ modulo $3$, then
\begin{equation}\label{eq:one-label}
\nu_h^{(1)}=
\frac{d^{h+1}-d^{r+1}}{d^2+d+1}+
\begin{cases}
0,&r=0,\\
1,&r=1,\\
d-1,&r=2.
\end{cases}
\end{equation}
In particular, a complete rooted tree of even branching degree has exactly one odd dominating set.
\end{corollary}
\begin{proof}
With all labels zero, \eqref{eq:complete-table} gives $L_h=\ell_1$ for even $h$ and $L_h=\ell_2$ for odd $h$. Let $\delta_h=[h\text{ is even}]$. From \eqref{eq:complete-k},
\[
\nu_{h+1}^{(0)}=d\nu_h^{(0)}-\delta_h+\delta_{h+1}
=d\nu_h^{(0)}+(-1)^{h+1}.
\]
The expression in \eqref{eq:zero-label} satisfies this recurrence and the initial condition $\nu_0^{(0)}=1$.

With all labels one and $d$ even, the state remains $\ell_0$ and $k_h=0$ for all $h$, so the nullity vanishes. If $d$ is odd, the state cycle is
\[
L_{3t}=\ell_0,\qquad L_{3t+1}=\ell_1,\qquad L_{3t+2}=\ell_2.
\]
For $d>1$, put $K_t=k_{3t}$. Equation~\eqref{eq:complete-k} gives
\[
K_{t+1}=d^3K_t+d(d-1),\qquad K_0=0,
\]
and hence
\[
K_t=\frac{d(d-1)(d^{3t}-1)}{d^3-1}
=\frac{d(d^{3t}-1)}{d^2+d+1}.
\]
Furthermore, $k_{3t+1}=dK_t$ and $k_{3t+2}=d^2K_t+d-1$. Adding $[L_h=\ell_1]$ gives \eqref{eq:one-label}. For $d=1$, \eqref{eq:complete-k} gives $k_h=0$; the same state cycle shows that the nullity is $1$ exactly when $h\equiv1\pmod3$, also as stated in \eqref{eq:one-label}. The final assertion follows from the number $2^{\nu_h^{(1)}}$ of odd dominating sets.
\end{proof}

\begin{remark}[Paths]\label{rem:paths}
For $d=1$ and arbitrary binary labels, \eqref{eq:complete-k} gives $k_h=0$ for every $h$. Thus $\nu_h=[L_h=\ell_1]\le1$. If the labels are eventually periodic, the finite-state argument in Theorem~\ref{thm:periodic} proves that the nullity is eventually periodic as well. This is sharper than merely describing it as affine on eventual residue classes.
\end{remark}

\subsection*{Declarations}
\noindent\textbf{Funding.} The author's research is partially supported by an AMS--Simons Research Enhancement Grant (award no.\ SFI-MPS-T-Institutes-00021934).

\smallskip\noindent\textbf{Conflict of interest.} The author declares that he has no conflict of interest.

\smallskip\noindent\textbf{Data and code availability.} No external datasets were used. The exact-arithmetic verification script and its validation summary are provided as supplementary material.


\clearpage
\begin{thebibliography}{10}
\bibitem{AminClarkSlater1998}
A.~T. Amin, L.~H. Clark, and P.~J. Slater,
Parity dimension for graphs,
\emph{Discrete Math.} \textbf{187} (1998), no.~1--3, 1--17.
\href{https://doi.org/10.1016/S0012-365X(97)00242-2}{doi:10.1016/S0012-365X(97)00242-2}.

\bibitem{AminSlaterZhang2002}
A.~T. Amin, P.~J. Slater, and G.-H. Zhang,
Parity dimension for graphs---a linear algebraic approach,
\emph{Linear Multilinear Algebra} \textbf{50} (2002), no.~4, 327--342.
\href{https://doi.org/10.1080/0308108021000049293}{doi:10.1080/0308108021000049293}.

\bibitem{Batal2022}
A. Batal,
Parity of an odd dominating set,
\emph{Commun. Fac. Sci. Univ. Ank. Ser. A1 Math. Stat.} \textbf{71} (2022), no.~4, 1023--1028.
\href{https://doi.org/10.31801/cfsuasmas.1051208}{doi:10.31801/cfsuasmas.1051208}.

\bibitem{Filmus2010}
Y. Filmus,
\emph{Range of symmetric matrices over GF(2)}, unpublished manuscript, January 2010.
\url{https://yuvalfilmus.cs.technion.ac.il/Manuscripts/Range.pdf}.

\bibitem{Goldwasser1997}
J. Goldwasser, W. Klostermeyer, and G. Trapp,
Characterizing switch-setting problems,
\emph{Linear Multilinear Algebra} \textbf{43} (1997), no.~1--3, 121--135.
\href{https://doi.org/10.1080/03081089708818520}{doi:10.1080/03081089708818520}.

\bibitem{HatzlWagner2008}
J. Hatzl and S. Wagner,
Combinatorial properties of a general domination problem with parity constraints,
\emph{Discrete Math.} \textbf{308} (2008), no.~24, 6355--6367.
\href{https://doi.org/10.1016/j.disc.2007.12.005}{doi:10.1016/j.disc.2007.12.005}.

\bibitem{Meyer2018}
C.~D. Meyer,
Rank my update, please,
\emph{Amer. Math. Monthly} \textbf{125} (2018), no.~1, 61--64.
\href{https://doi.org/10.1080/00029890.2017.1389199}{doi:10.1080/00029890.2017.1389199}.

\bibitem{Minevich2012}
I. Minevich,
\emph{Symmetric matrices over $\mathbb F_2$ and the Lights Out problem},
arXiv:1206.2973, 2012.
\url{https://arxiv.org/abs/1206.2973}.

\bibitem{Sutner1989}
K. Sutner,
Linear cellular automata and the Garden-of-Eden,
\emph{Math. Intelligencer} \textbf{11} (1989), no.~2, 49--53.
\href{https://doi.org/10.1007/BF03023823}{doi:10.1007/BF03023823}.

\bibitem{Wagner2010}
S. Wagner,
Counting all parity realizable trees,
\emph{J. Combin. Math. Combin. Comput.} \textbf{75} (2010), 11--32.
\end{thebibliography}
\end{document}